\documentclass[11pt,leqeq]{article}
\usepackage{amssymb,amsthm}
\usepackage{amsmath}
\usepackage{amscd}
\usepackage{xy}
\usepackage[dvips]{graphicx}
\newtheorem{thm}{Theorem}
\newtheorem{lema}[thm]{Lemma}
\newtheorem{cor}[thm]{Corollary}
\newtheorem{prop}[thm]{Proposition} 

\date{}
\begin{document}
\title{On the $k$-gamma $q$-distribution}
\author{Rafael D\'\i az, Camilo Ortiz, and Eddy Pariguan}
\maketitle
\begin{abstract}
We provide combinatorial as well as probabilistic interpretations for the
$q$-analogue of the  Pochhammer $k$-symbol introduced by D\'iaz and Teruel.
We introduce $q$-analogues of the Mellin transform in order to study the $q$-analogue of the
$k$-gamma distribution.

\end{abstract}

\section{Introduction}

There is a general strategy for building bridges between combinatorics and
measure theory which we describe.  Let $d\mu $ be a measure
on a interval $I \subseteq [0, \infty) \subseteq \mathbb{R}$.
We say that $d\mu$ is a combinatorial measure  if for each $n \in \mathbb{N}$ the $n$-th moment
 of $d\mu$
is a non-negative integer. Equivalently, let $M_{d\mu}$ be the Mellin
transform of $d\mu$ given by
$$M_{d\mu}(t)=\int_I x^{t-1}d\mu .$$ Then $d\mu$ is a combinatorial measure
 if and only if $M_{d\mu}(n)\in \mathbb{N}$ for $n \in \mathbb{N}_+.$\\

Let $cmeas$ be the set of combinatorial measures, and consider the map
$m:cmeas \rightarrow \mathbb{N}^{\mathbb{N}}$ that sends a combinatorial
measure into its moment's sequence $(m_0, ..., m_n,...).$
Recall \cite{DZeilberger} that a sequence of finite sets $(s_0,...,s_n,...)$
provides a combinatorial interpretation for a sequence of integers $(m_0,...,m_n,...)$
if it is such that $|s_n|=m_n.$ By analogy we say that the sequence of
finite sets $(s_0,...,s_n,...)$ provides a combinatorial interpretation
for a measure  $d\mu$ if for each $n \in  \mathbb{N}$ the
following identity holds: $$|s_n|=\int_I x^n d\mu.$$
In this work we consider the reciprocal problem: given  $m=(m_0,...,m_n,...) \in \mathbb{N}$
find a combinatorial interpretation for it, and furthermore find a
combinatorial measure $d\mu$ such that its sequence of
 moments is $m$. Our main goal is to establish an instance
 of the correspondence combinatorics/measure theory described above within the context
of $q$-calculus.
Namely, we are going to study the combinatorial and the measure theoretic
 interpretations for the   $k$-increasing factorial $q$-numbers
$[1]_{n,k}= [1]_q[1+k]_q[1+2k]_q. \ . \ .[1+(n-1)k]_q \in \mathbb{N}[q]^{\mathbb{N}}$
which are obtained as an instance of the $q$-analogue of the  Pochhammer $k$-symbol given by
$$[t]_{n,k}= [t]_q[t+k]_q[t+2k]_q. \ . \ .[t+(n-1)k]_q =
\prod_{j=0}^{n-1}[t+jk]_q$$ where $[t]_q=\frac{1-q^t}{1-q}$ is the $q$-analogue of $t$.
The search for the combinatorial and measure theoretic interpretation for the $k$-increasing factorial $q$-numbers must be made within the context of $q$-calculus; this means that we have to broaden our techniques in order to include $q$-combinatorial interpretations, and the $q$-analogues for the Lebesgue's measure
and the Mellin's transform.

\section{The $k$-gamma measure}

Perhaps the best known example of the relation combinatorics/measure theory discussed in the introduction
comes from the factorial numbers $n!$ which count, respectively, the number of elements of $S_n$, the group of permutations of a set with $n$ elements. The Mellin transform of the measure $e^{-x}dx$ is the classical gamma function
given for $t>0$ by
$$\Gamma(t)=\int_{0}^{\infty} x^{t-1} e^{-x} dx .$$
The moments of the measure $e^{-x}dx$ are precisely the factorial numbers, indeed we have that
$$|S_n|= n!=\Gamma(n+1)=\int_{0}^{\infty} x^{n} e^{-x} dx.$$
Notice that  $n!=(1)_n,$ where the Pochhammer symbol $(t)_n$ is given by
$$(t)_n =t(t+1)(t+2). \ . \ .(t+(n-1)).$$

As a second example \cite{DP} consider the combinatorial and measure theoretical interpretations for the $k$-increasing factorial numbers $(1)_{n,k}= (1+k)(1+2k). \ . \ .(1+(n-1)k),$
which arise as an instance of the Pochhammer $k$-symbol given by
$$(t)_{n,k}= t(t+k)(t+2k). \ . \ .(t+(n-1)k)=\prod_{j=0}^{n-1}(t+jk).$$
The combinatorics of the Pochhammer $k$-symbol has attracted considerable attention in the literature,
from the work of Gessel and Stanley \cite{ge}  up to the quite recent works \cite{ca, ku}. Assume
$t$ is a non-negative integer and let $\mathrm{T}_{n,k}^t$ be the set of isomorphisms classes of planar rooted trees $T$ such that:
1) The set of internal vertices, i.e. vertices with one outgoing edge and at least one incoming edge, of $T$ is $\{1,2,....,n \}$;
2) $T$  has a unique vertex with no outgoing edges called the root; $T$ has a set $L(T)$  of vertices called leaves, the leaves have
no incoming edges;
3) The valence of each internal vertex of $T$ is $k+2;$ 4) The valence of the root is $t$;
5) If the internal vertex $i$ is on the path from the internal vertex
$j$ to the root, then $i < j.$ \\

Note that the set of leaves $L(T)$ comes with a natural order, and thus we can assign a number between $1$ to $|L(T)|$ to
each leave.
Figure \ref{ejemp} shows an example of a graph in $\mathrm{T}_{4,2}^2$.\\

One can show by induction that $(t)_{n,k}=|\mathrm{T}_{n,k}^t|.$

\begin{figure}[h!]
  \centering
    \includegraphics[width=.18\textwidth]{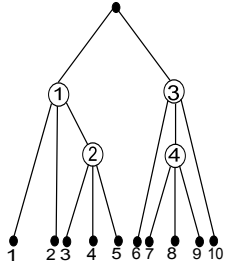}
    \caption{Example of a tree in $\mathrm{T}_{4,2}^2$.}
    \label{ejemp}
\end{figure}

The Mellin transform of the measure $e^{\frac{-x^k}{k}}dx$ is the $k$-gamma function $\Gamma_k$ given for $t > 0$ by
$$\Gamma_k(t)=\int_{0}^{\infty} x^{t-1} e^{-\frac{x^k}{k}} dx  .$$
The $k$-gamma function $\Gamma_k : (0, \infty ) \longrightarrow \mathbb{R}$ is univocally determined \cite{DP} by
the following properties:  $\Gamma_k(t+k)=t\Gamma_k(t)$ for $t\in \mathbb{R^+};$
$\Gamma_k(k)=1;$
$\Gamma_k $ is logarithmically convex. See \cite{ko, ma} for further properties of the $k$-gamma function.\\

The $k$-increasing factorial numbers appear as moments of the $\Gamma_k$ function as follows:
$$|\mathrm{T}_{n,k}^1|=(1)_{n,k}=\Gamma_k(1+ nk)=  \frac{1}{\Gamma_k(1)} \int_{0}^{\infty} x^{nk} e^{\frac{-x^k}{k}} dx=
 \frac{k^{\frac{k-1}{k}}}{\Gamma(\frac{1}{k})} \int_{0}^{\infty} x^{nk} e^{\frac{-x^k}{k}} dx.$$
Indeed the following more general identity  holds:
$$|\mathrm{T}_{n,k}^t| = (t)_{n,k} =\frac{\Gamma_k(t+nk)}{\Gamma_k(t)}=
\frac{1}{\Gamma_k(t)}\int_{0}^{\infty} x^{t+nk-1} e^{\frac{-x^k}{k}} dx.$$

\section{Review of $q$-calculus}\label{Basic}

In this section we introduce some useful basic definitions \cite{RA, Ch, HY}. We
begin introducing the $q$-derivative and the Jackson $q$-integral.
Let $\mathrm{Map}(\mathbb{R},\mathbb{R})$ be the real vector
space of functions from $\mathbb{R}$ to $\mathbb{R}$. Fix a real number
$0 \leq q<1,$  the $q$-derivative is the linear operator
$$\partial_q:\mathrm{Map} (\mathbb{R},\mathbb{R})\rightarrow
\mathrm{Map}(\mathbb{R}\setminus \{0\},\mathbb{R}) \ \ \ \ \ \mbox{given by}$$ $$\partial_q f(x)=\frac{f(qx)-
f(x)}{(q-1)x}. \ \ \ \ \mbox{For example we have that}\ \  \partial_0 f(x)=\frac{f(x)-
f(0)}{x}.$$  Notice that $\partial_q f$ is not a priori well-defined at $x=0$. Nevertheless, it is often the
case that $\partial_q f$ can be extended by continuity over the whole real line, e.g.
when $f$ is a polynomial function.\\

For $0\leq a < b \leq +\infty$ the Jackson $q$-integral from $a$ to $b$ of  $f \in
\mathrm{Map}(\mathbb{R},\mathbb{R})$  is given  by
$$\int_{a}^{b}f(x)d_qx=(1-q)b\sum_{n=0}^{\infty}q^nf(q^nb)-(1-q)a\sum_{n=0}^{\infty}q^nf(q^na).$$
For example we have that
$$\int_a^{b} f(x) d_0x =bf(b)-af(a).$$

Set $I_q f(x)=f(qx).$ The following properties hold for $f,g \in \mathrm{Map}(\mathbb{R},\mathbb{R}):$
\begin{eqnarray*}
\partial_q(fg)&=& \partial_qf  g + I_q  f \partial_q g \\
\partial_q(f(a x^{b}))&=& a[b]_qx^{b-1}\partial_{q^{b}}f(ax^{b})  \\
f(b)g(b)-f(a)g(a)&=& \int_a^b \partial_qf  gd_qx + \int_a^b I_q g \partial_q fd_qx,
\end{eqnarray*}

For $0 <q<1$, $x,y \in\mathbb{R}$, $n \in \mathbb{N}_+$, and
$t \in \mathbb{R}$ we set
 $$(x+y)_{q^k}^n = \prod_{j=0}^{n-1} (x+ q^{jk}y), \ \ \ (x+y)_{q^k}^{\infty} = \prod_{j=0}^{\infty} (x+ q^{jk}y) \mbox{\ \ \  and \ \ \ }
(1+x)_{q^k}^{t}=\frac{(1+x)_{q^k}^{\infty}}{(1+q^{kt}x)_{q^k}^{\infty}}.$$

\section{$q$-Analogue of the $k$-gamma function}

We proceed to study the $q$-analogue of the $k$-increasing factorial numbers
$$[1]_{n,k}= [1]_q[1+k]_q[1+2k]_q. \ . \ .[1+(n-1)k]_q $$
which are an instance of the $q$-analogue of the Pochhammer $k$-symbol $[t]_{n,k}$ given
for $t \in \mathbb{R}$ by
$$[t]_{n,k}= [t]_q[t+k]_q[t+2k]_q. \ . \ .[t+(n-1)k]_q =
\prod_{j=0}^{n-1}[t+jk]_q.$$

The motivation behind our definition of the $q$-analogue of the $k$-gamma function comes from the
work of De Sole and Kac \cite{So}, where they introduced a
$q$-deformation of the gamma function given by the $q$-integral:
$$\Gamma_{q}(t)=\int_{0}^{\frac{1}{1-q}}x^{t-1}E_{q}^{-q
x}d_qx,$$
where the $q$-analogue $E_{q}^{x}$ of the exponential function is given by
$$E_{q}^{x}= \sum_{n=0}^{\infty}q^{\frac{n(n-1)}{2}}\frac{x^n}{[n]_{q}!}.$$ For example we have that $ E_0^x=1 +x,$ $E_0^{-0x}=1$, and therefore $\Gamma_0(t)=1.$\\

We define the $q$-analogue of the $k$-gamma function $\Gamma_{q,k}$ by
demanding that it satisfies the $q$-analogues of the properties
of the $\Gamma_k$ function. Thus $\Gamma_{q,k}$ is such that
$\Gamma_{q,k}(t+k)=[t]_q \Gamma_{q,k}(t)$ and $\Gamma_{q,k}(k)=1.$
Several applications of the former property show that $$\Gamma_{q,k}(nk)=\prod_{j=1}^{n-1}[jk]_q=
 \prod_{j=1}^{n-1} \frac{(1-q^{jk})}{{(1-q)}}
=\frac{(1-q^k)_{q^k}^{n-1}}{(1-q)^{n-1}}.$$
After a change of variables the function $\Gamma_{q,k}$  may be written as follows:
$$\Gamma_{q,k}(t)=\frac{{(1-q^k)_{q^k}^{{\frac{t}{k}}
-1}}}{{(1-q)^{\frac{t}{k}-1}}}. \ \ \ \ \mbox{For example we have that} \ \   \Gamma_{0,k}(t)=1.$$ The previous formula
implies an infinite product expression for $\Gamma_{q,k}$ given by
$$\Gamma_{q,k}(t)=\frac{(1-q)^{1-\frac{t}{k}}(1-q^k)_{q^k}^{\infty}}{(1-q^t)_{q^k}^{\infty}},$$ and also  the following result.

\begin{lema}
{\em The $q,k$-gamma function $\Gamma_{q,k}$ and the $q^k$-gamma  function  $\Gamma_{q^k}$ are related by the identity
$\ \ \Gamma_{q,k}(t) = [k]_q^{\frac{t}{k}-1}\Gamma_{q^k}(\frac{t}{k}) .$
}
\end{lema}

The following result \cite{CTT} provides an integral representation for $\Gamma_{q,k}$.

\begin{prop}
$$\Gamma_{q,k}(t)=\int_{0}^{\frac{[k]_q^{\frac{1}{k}}}{(1-q^k)^{\frac{1}{k}}}}
x^{t-1}E_{q^k}^{-\frac{q^k x^k}{[k]_q}}d_qx.$$
\end{prop}

This integral representation for $\Gamma_{q,k}$ may be regarded as a  $q$-analogue of the
Mellin transform, therefore one is entitled to consider the $q$-measure
$$E_{q^k}^{-\frac{q^k x^k}{[k]_q}}d_qx \ \ \ \mbox{defined on the interval} \ \ \
\big[ 0, \frac{[k]_q^{\frac{1}{k}}}{(1-q^k)^{\frac{1}{k}}} \big]$$
as the inverse Mellin $q$-transform of the $\Gamma_{q,k}$ function. Figure \ref{eje} shows
the graph of $E_{q^k}^{-\frac{q^k x^k}{[k]_q}}$ for  $q=0.6$ and $1\leq k \leq 5$.

\begin{figure}[h!]\label{eje}
\centering
\includegraphics[width=10cm,height=5cm]{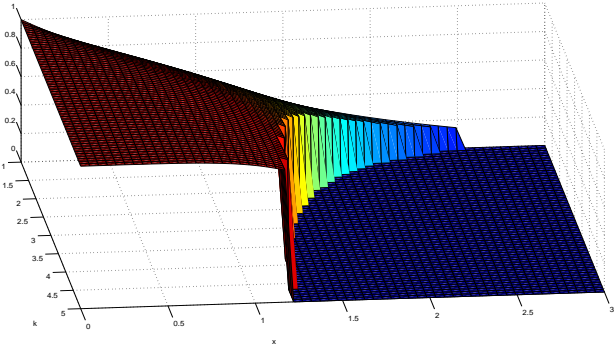}
\caption{Display of $E_{q^k}^{-\frac{q^k x^k}{[k]_q}}$ for $q=0.6$ and $1\leq k \leq 5$.}
\end{figure}

One can check that for $0 \leq x \leq  \frac{[k]_q^{\frac{1}{k}}}{(1-q^k)^{\frac{1}{k}}}$
 the function $E_{q^k}^{-\frac{q^k x^k}{[k]_q}}$ is given by
$$E_{q^k}^{-\frac{q^k x^k}{[k]_q}}=
 \sum_{n=0}^{\infty}\frac{(-1)^n q^{\frac{kn(n+1)}{2}}x^{kn}}{[k]_q^n [n]_{q^k}! }.$$

\begin{thm}\label{ya}
{\em The function $\Gamma_{q,k}$ is given by
$$ \Gamma_{q,k}(t)=(1-q)^{1-\frac{t}{k}}
\sum_{n=0}^{\infty} \frac{ q^{ \frac{kn(n+1)}{2} } }{ (1-q^{kn+t})(q^k -1)^n[n]_{q^k} ! } .$$}
\end{thm}

\begin{proof}
From Theorem \ref{cool} below we know that
$$\int_{0}^{x}
s^{t-1}E_{q^k}^{-\frac{q^k s^k}{[k]_q}}d_qs =
(1-q)x^t\sum_{n=0}^{\infty}\frac{(-1)^n q^{\frac{kn(n+1)}{2}} x^{kn}  }
{(1-q^{kn+t}) [k]_q^n [n]_{q^k}! }.$$
The desired result follows taking $x=\frac{[k]_q^{\frac{1}{k}}}{(1-q^k)^{\frac{1}{k}}}$.
\end{proof}

\begin{cor}
{\em
$$(1-q^k)_{q^k}^{\frac{t}{k}} = \sum_{n=0}^{\infty} \frac{ q^{ \frac{kn(n+1)}{2} } }{ (1-q^{kn+t})(q^k -1)^n[n]_{q^k} ! }.$$}
\end{cor}

\begin{proof}
Follows from Theorem \ref{ya} and the identity $\Gamma_{q,k}(t)=\frac{{(1-q^k)_{q^k}^{{\frac{t}{k}}
-1}}}{{(1-q)^{\frac{t}{k}-1}}}.$
\end{proof}

By definition the cumulative distribution function  associated with the measure $$E_{q^k}^{-\frac{q^k x^k}{[k]_q}}d_qx \ \ \ \ \mbox{is given for} \ \ \ 0 \leq x \leq  \frac{[k]_q^{\frac{1}{k}}}{(1-q^k)^{\frac{1}{k}}}\ \ \ \ \mbox{by} \ \ \ \int_{0}^{x}
E_{q^k}^{-\frac{q^k s^k}{[k]_q}}d_qs.$$

\begin{prop}
{\em $$\int_{0}^{x}
E_{q^k}^{-\frac{q^k s^k}{[k]_q}}d_qs=(1-q)x
\sum_{n=0}^{\infty}\frac{(-1)^n q^{\frac{kn(n+1)}{2}}x^{kn+1}}{(1-q^{kn+1}) [k]_q^n [n]_{q^k}!  }.$$}
\end{prop}

\begin{proof}

The result follows from Theorem \ref{cool} below taking $t=1.$

\end{proof}

\section{Combinatorial interpretation of the Pochhammer $q,k$-symbol}

Just as in combinatorics one studies the cardinality of finite sets, in
$q$-combinatorics one studies the cardinality of $q$-weighted finite sets, i.e.
pairs $(x,\omega)$ where $x$ is a finite set and the $q$-weight is an arbitrary map
$\omega:x \rightarrow \mathbb{N}[q]$ from $x$ to $\mathbb{N}[q]$ the algebra of polynomials in
$q$ with non-negative integer coefficients. The
cardinality of the pair $(x, \omega)$ is by definition given by
$$|x, \omega|=\sum_{i\in x}\omega(i) \in \mathbb{N}[q].$$
To provide a $q$-combinatorial interpretation for the Pochhammer $k$-symbol $[t]_{n,k}$ we
let again $t$ be a positive integer and consider the set $\mathrm{T}_{n,k}^t$ of planar rooted trees introduced above.
Next we define a $q$-weight $\omega$ on $\mathrm{T}_{n,k}^t$. The construction of $\omega$
is based on the following elementary facts:\\

\begin{figure}[h!]
  \centering
    \includegraphics[width=.30\textwidth]{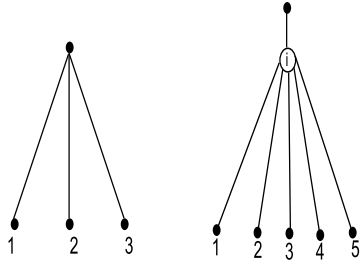}
    \caption{Display of the tree $r_3$ and a tree $c_i$ with $5$ leaves.}
    \label{nuevo}
\end{figure}

\noindent 1)  Let $r_t$ be the rooted tree with $t$ leaves and no internal vertices. See Figure 3. \\

\noindent 2) For $1 \leq i \leq n$, let $c_i$ be the rooted tree with $i$ as its unique internal vertex and $k+1$ leaves.
See Figure 3.\\

\noindent 3) If $T$ is a planar rooted tree and $l$ is a number between $1$ and $|L(T)|$, then
there is a well-defined rooted planar tree $T \circ_l c_i$ obtained by gluing the root of $c_i$ with the leave $l$
of $T$ to form a new edge.\\

\noindent 4) Clearly each tree $T \in \mathrm{T}_{n,k}^t$ can be written in a unique way as
$$T=(...((r_t \circ_{l_1} c_1)\circ_{l_2} c_2)...)\circ_{l_{n}} c_n .$$

\noindent 5) The weight $\omega(T)$ of a tree $T$ written in the form above is given by
$$\omega(T)=\prod_{i=1}^{n-1} q^{l_i -1} \in \mathbb{N}[q].$$

For the tree $T$ from Figure 1 we have that
$$T = ((( r_2 \circ_1 c_1) \circ_3 c_2 ) \circ_6 c_3) \circ_7 c_4 \ \ \mbox{ and }\ \ \omega(T) = q^0 q^2 q^5 q^6 = q^{13}.$$

\begin{thm}{\it
$$[t]_{n,k} = |\mathrm{T}_{n,k}^t, \omega | .$$
}
\end{thm}

\begin{proof}
The proof goes by induction on $n$.  We have the following chain of identities
$$|\mathrm{T}_{n+1,k}^t, \omega | = \sum_{T \in \mathrm{T}_{n+1, k}^t}\omega(T)=
\sum_{S\in \mathrm{T}_{n,k}^t }\sum_{l \in L(S)}\omega(S \circ_l c_{n+1})=
\sum_{S\in \mathrm{T}_{n,k}^t }\sum_{l=1}^{t+kn}\omega(S)q^{l-1}=$$
$$=\left(\sum_{S\in \mathrm{T}_{n,k}^t}\omega(S) \right) \left(\sum_{l=1}^{t+nk}q^{l-1} \right)
=|\mathrm{T}_{n+1,k}^t, \omega |[t+nk]_q =  [t]_{n,k}[t+nk]_q=[t]_{n+1,k}.$$
In the computation above we used two main facts: 1) Each tree $\mathrm{T}_{n,k}^t$ has exactly
$t +nk$ leaves; 2) Each tree $T \in \mathrm{T}_{n+1, k}^t$ can be written in a unique way
as $T= S \circ_l c_{n+1}$ where $S \in \mathrm{T}_{n, k}^t$, $l$ is a leaf of $S$, and
$c_{n+1}$ is the rooted tree with $k+1$ leaves and $n+1$ as its unique internal vertex.

\end{proof}

\section{$k$-Gamma $q$-distribution}

We are ready to define the $k$-gamma $q$-distribution. From the identity
$$\Gamma_{q,k}(t)=\int_{0}^{\frac{[k]_q^{\frac{1}{k}}}{(1-q^k)^{\frac{1}{k}}}}
x^{t-1}E_{q^k}^{-\frac{q^k x^k}{[k]_q}}d_qx$$
we see that the function
$$ x^{t-1}\frac{E_{q^k}^{-\frac{q^k x^k}{[k]_q}}}{\Gamma_{q,k}(t)}$$ defines a $q$-density
on the interval $[0,  \frac{[k]_q^{\frac{1}{k}}}{(1-q^k)^{\frac{1}{k}}} ]$, in the sense that it is a non-negative function whose $q$-integral is equal to one. Consider the case $t=1$ an $k=3$.
Figure \ref{dng} shows the graph of $ \frac{E_{q^3}^{-\frac{q^3 x^3}{[3]_q}}}{\Gamma_{q,3}(1)}$ for $q\in [0,1)$.

\begin{figure}[h!]\label{dng}
\centering
\includegraphics[width=10cm,height=5cm]{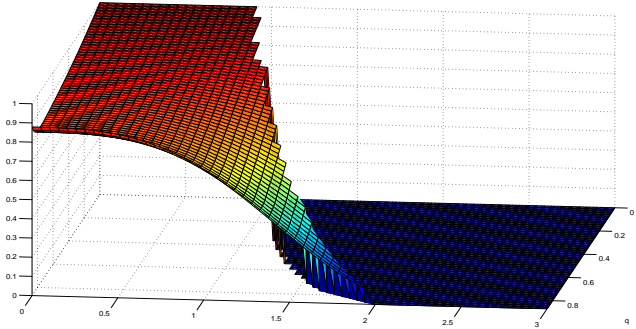}
\caption{Display of $ \frac{E_{q^3}^{-\frac{q^3 x^3}{[3]_q}}}{\Gamma_{q,3}(1)}$ for $q \in [0, 1)$.}
\end{figure}

\begin{thm}\label{cool}
{\em The cumulative distribution of the $k$-gamma $q$-density is given by
$$\frac{1}{\Gamma_{q,k}(t)}\int_{0}^{x}
s^{t-1}E_{q^k}^{-\frac{q^k s^k}{[k]_q}}d_qs =
\frac{(1-q)x^t}{\Gamma_{q,k}(t)}\sum_{n=0}^{\infty}\frac{(-1)^n q^{\frac{kn(n+1)}{2}} x^{kn}  }
{[k]_q^n [n]_{q^k}! (1-q^{kn+t}) }.$$}
\end{thm}

\begin{proof}

\begin{eqnarray*}
\frac{1}{\Gamma_{q,k}(t)} \int_{0}^{x}
s^{t-1}E_{q^k}^{-\frac{q^k s^k}{[k]_q}}d_qs
&=& \frac{(1-q)x}{\Gamma_{q,k}(t)} \sum_{m=0}^{\infty} q^m \sum_{n=0}^{\infty}\frac{(-1)^n q^{\frac{kn(n+1)}{2}}(q^m x)^{kn +t-1}}{[k]_q^n [n]_{q^k}!}     \\
&=& \frac{1-q}{\Gamma_{q,k}(t)}
\sum_{n=0}^{\infty} \frac{(-1)^n q^{\frac{kn(n+1)}{2}}x^{kn + t}}{[k]_q^n [n]_{q^k}! }
\sum_{m=0}^{\infty} q^{m(kn+t)}\\
&=&  \frac{1-q}{\Gamma_{q,k}(t)} \sum_{n=0}^{\infty} \frac{ (-1)^n q^{ \frac{kn(n+1)}{2} } x^{kn+t}}{ [k]_q^n [n]_{q^k}!(1-q^{kn+t}) }
\end{eqnarray*}

\end{proof}
 Consider the case $t=1$ an $k=3$.Figure \ref{dng2} shows the cumulative distribution associated to the $q$-density $ \frac{E_{q^3}^{-\frac{q^3 x^3}{[3]_q}}}{\Gamma_{q,3}(1)}$ for $q\in [0,1)$.

\begin{figure}[h!]\label{dng2}
\centering
\includegraphics[width=10cm,height=5cm]{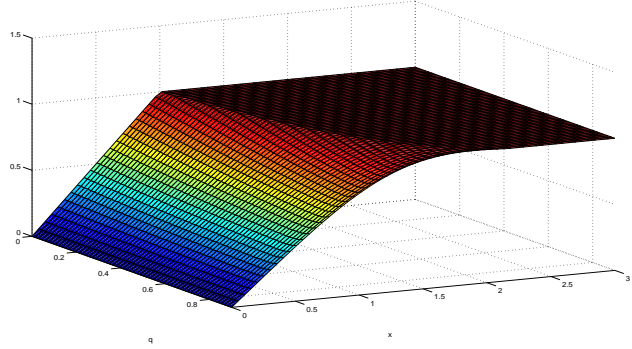}
\caption{Cumulative distribution of the $q$-density $ \frac{E_{q^3}^{-\frac{q^3 x^3}{[3]_q}}}{\Gamma_{q,3}(1)}$ for $q \in [0, 1)$.}
\end{figure}

The previous considerations imply our next result which establishes an example of the link between $q$-combinatorics and $q$-measure theory
promised in the introduction.

\begin{thm} {\em The $k$-increasing factorial $q$-numbers appear as moments of the $\Gamma_{q,k}$ function as follows: $$|\mathrm{T}_{n,k}^1, \omega| = [1]_{n,k}=\Gamma_{q,k}(1+ nk)=\frac{1}{\Gamma_{q,k}(1)}\int_{0}^{\frac{[k]_q^{\frac{1}{k}}}{(1-q^k)^{\frac{1}{k}}}}
x^{nk}E_{q^k}^{-\frac{q^k x^k}{[k]_q}}d_qx.$$
Indeed the following more general identity  also holds:
$$|\mathrm{T}_{n,k}^t, \omega | = [t]_{n,k} =
\frac{\Gamma_{q,k}(t+nk)}{\Gamma_{q,k}(t)}= \frac{1}{\Gamma_{q,k}(t)}\int_{0}^{\frac{[k]_q^{\frac{1}{k}}}{(1-q^k)^{\frac{1}{k}}}}
x^{t+nk-1}E_{q^k}^{-\frac{q^k x^k}{[k]_q}}d_qx.$$}
\end{thm}

\section{ $k$-Beta $q$-distribution}
Recall that the classical beta function is given for $s,t >0$ by
$$B(t,s)=  \frac{\Gamma(t)\Gamma(s)}{\Gamma(t+s)}=\int_{0}^{1} x^{t-1}(1-x)^{s-1}dx .$$
The $q$-analogue of the $k$-beta function is correspondingly defined by
$$B_{q,k}(t,s)=\frac{\Gamma_{q,k}(t)\Gamma_{q,k}(s)}{\Gamma_{q,k}(t+s)}=
\frac{(1-q)(1-q^k)_{q^k}^{\frac{s}{k}-1}}{(1-q^t)_{q^k}^{\frac{s}{k}}}.$$
Notice that $B_{0,k}(t,s)= 1$.  One can show \cite{CTT} that the function $B_{q,k}$ has the following integral representation
$$B_{q,k}(t,s)=[k]_q^{-\frac{t}{k}}\int_{0}^{[k]_{q}^{\frac{1}{k}}}x^{t-1}\left(
1-q^k\frac{x^k}{[k]_q}\right)_{q^k}^{\frac{s}{k}-1}d_qx.$$
Because of the factor $[k]_q^{-\frac{t}{k}}$ this integral representation is not quite a Mellin
transform. However we see that the $q$-measure $$\left(
1-q^k\frac{x^k}{[k]_q}\right)_{q^k}^{\frac{s}{k}-1}d_qx$$ is a Mellin $q$-transformation inverse
of the function $$B_{q,k}(t,s)[k]_q^{\frac{t}{k}} .$$

On the other hand we see that the function
$$\frac{x^{t-1}\left(
1-q^k\frac{x^k}{[k]_q}\right)_{q^k}^{\frac{s}{k}-1}}{B_{q,k}(t,s)[k]_q^{\frac{t}{k}}} $$
defines a $q$-density on the interval $[\ 0, \ [k]_q^{\frac{1}{k}} \ ],$ indeed it defines
a $q$-analogue for the $k$-beta density. Figure \ref{dng3} below shows the graph of the $q$-density

$$\frac{x^{-0.5}\left(
1-q^3\frac{x^3}{[3]_q}\right)_{q^3}^{\frac{0.5}{3}-1}}{B_{q,3}(0.5,0.5)[3]_q^{\frac{0.5}{3}}} $$

\begin{figure}[h!]\label{dng3}
\centering
\includegraphics[width=10cm,height=5cm]{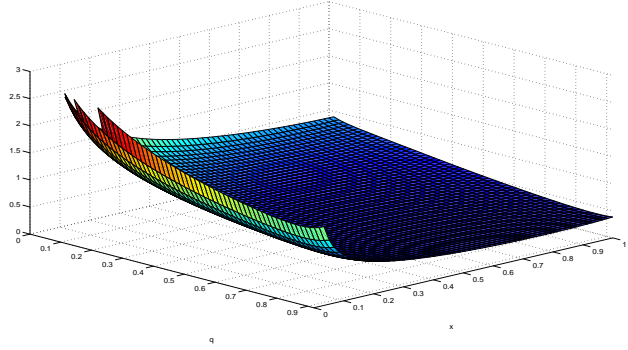}
\caption{Display of the function $\frac{x^{-0.5}\left(
1-q^3\frac{x^3}{[3]_q}\right)_{q^3}^{\frac{0.5}{3}-1}}{B_{q,3}(0.5,0.5)[3]_q^{\frac{0.5}{3}}} $ for $0 \leq q < 1$ and $0 \leq x \leq 1$.}
\end{figure}

Our final result provides an explicit formula for the cumulative beta $q$-distribution. If follows as an easy consequence of the
definition of the Jackson integral and the definition of $B_{q,k}(t,s)$.

\begin{thm}{\em
The cumulative beta $q$-distribution is given by
$$\frac{1}{B_{q,k}(t,s)[k]_q^{\frac{t}{k}}}\int_0^x s^{t-1}\left(
1-q^k\frac{s^k}{[k]_q}\right)_{q^k}^{\frac{s}{k}-1} d_qx =
\frac{(1-q)x^t}{B_{q,k}(t,s)[k]_q^{\frac{t}{k}}}\sum_{n=0}^{\infty}q^{nt}\left(
1-q^{k(n+1)}\frac{x^k}{[k]_q}\right)_{q^{k}}^{\frac{s}{k}-1} .$$}
\end{thm}

\begin{figure}[h!]\label{dng4}
\centering
\includegraphics[width=10cm,height=5cm]{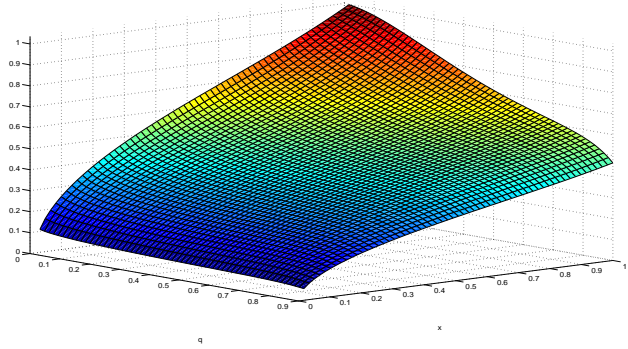}
\caption{Cumulative distribution of the $q$-density $\frac{x^{-0.5}\left(
1-q^3\frac{x^3}{[3]_q}\right)_{q^3}^{\frac{0.5}{3}-1}}{B_{q,3}(0.5,0.5)[3]_q^{\frac{0.5}{3}}} $ for $0 \leq q < 1$ and $0 \leq x \leq 1$.}
\end{figure}

\noindent ragadiaz@gmail.com\\
\noindent Instituto de Matem\'aticas y sus Aplicaciones, Universidad Sergio Arboleda, Bogot\'a, Colombia \\

\noindent camiloortiz@javeriana.edu.co,\ \ \ \ epariguan@javeriana.edu.co \\
Departamento de Matem\'aticas, Pontificia Universidad Javeriana,
Bogot\'a, Colombia

\end{document}